\documentclass[11pt]{amsart} 

\usepackage[utf8]{inputenc} 
\usepackage{amsmath}
\usepackage{amssymb}
\usepackage{mathtools}
\usepackage{mdwlist}
\usepackage{array}
\usepackage{url}

\def\scr{\mathcal}

\def\ditemfirst#1#2{\begin{enumerate}\item \label{#1} #2\suspend{enumerate}}

\def\hR{{}^*\mathbb{R}}

\newtheorem{lem}{Lemma}
\newtheorem{prp}{Proposition}
\newtheorem{thm}{Theorem}
\newtheorem{cor}{Corollary}

\usepackage{graphicx} 

\usepackage{array} 

\def\St{\operatorname{St}}

\def\Dom{\operatorname{Dom}}
\def\Range{\operatorname{Range}}

\def\<{\langle}
\def\>{\rangle}

\def\Z{\mathbb Z}

\def\R{\mathbb R}

\def\star#1{{}^*#1}

\def\hR{\star{\R}}
\def\Power{\mathcal P}

\title[Invariance of Probabilities]{Non-Classical Probabilities Invariant Under Symmetries [corrected]}
\author{Alexander R. Pruss}

\begin{document}
\sloppy
\begin{abstract}
Classical real-valued probabilities come at a philosophical cost: in many infinite situations, they assign
the same probability value---namely, zero---to cases that are impossible as well as to cases that are possible. There are three
non-classical approaches to probability that can avoid this drawback: full conditional probabilities, qualitative probabilities
and hyperreal probabilities. These approaches have been criticized for failing to preserve intuitive symmetries that can
be preserved by the classical probability framework, but there has not been a systematic study of the conditions under which these
symmetries can and cannot be preserved. This paper fills that gap by giving complete characterizations under which symmetries
understood in a certain ``strong'' way can be preserved by these non-classical probabilities, as well as by offering some
results to make it plausible that the strong notion of symmetry here is the right one. Philosophical implications are briefly
discussed, but the main purpose of the paper is to offer technical results to help make further philosophical discussion
more sophisticated.

\textbf{This is a corrected version of a paper published in \textit{Synthese} and does not correspond to any published version.
The original version had an erroneous proof of what it called Lemma~2 (Lemma~3 in this version).}
\end{abstract}

\maketitle
\section{Introduction}
Bayesian epistemologists are attracted to the regularity condition on probability that says that every nonempty event
has non-zero probability. This condition should capture, for instance, the distinction between the possibility of a
sequence of fair coins all showing heads and the impossibility the coins in the sequence showing a square circle.
As is well-known, the regularity condition cannot be satisfied in classical probability theory when there are
uncountably many pairwise disjoint events~\cite{Hajek03}. However, multiple attempts have been made to replace the classical probabilistic
framework with one that allows for regularity or something like it. The three most prominent such frameworks are: (1)~full
conditional probabilities, (2)~qualitative or comparative probabilities and (3)~non-Archimedean or hyperreal probabilities.

These approaches have been criticized in the philosophical literature for failing to preserve intuitive symmetries of
probabilistic situations (e.g., \cite{Pruss13, PrussPopper, Williamson07}; there has also been a significant discussion of
these criticisms, e.g., \cite{BHW18, Howson17, Parker19}).
For instance, it is impossible for a rotation-invariant finitely-additive hyperreal probability
defined for all subsets of the unit circle to satisfy regularity for the reason
that if $x$ is any point on the circle and $\rho$ is a rotation by an irrational number of degrees,
then rotational invariance would require the
countable set $A = \{ x,\rho x,\rho^2x,... \}$ to have the same probability as $\rho A$,
but $\rho A=\{ \rho x, \rho^2 x, \rho^3 x, ... \}$ is a proper subset of $A$ and hence regularity would be
violated (cf.\ \cite{Blumenthal40, BW69, Pruss13}). This counterintuitive phenomenon of a set $A$ that can be rotated
to form a proper subset $\rho A$ is a junior relative of the Banach-Tarski paradox where a ball can be partitioned into
a finite number of pieces that can be reassembled into two balls.

But there is an interesting technical question that has not received much exploration in the qualitative and hyperreal
cases: under what exact conditions do there exist regular non-classical probabilities that are invariant under some group
of symmetries? The point of this paper is to give a complete answer to this question in the case of certain ``strong''
notions of invariance for all three main non-classical types of probability. The answer in all cases will involve the non-existence
of relatives of the Banach-Tarski paradox and of the above rotational paradox.

We will be primarily interested in probabilities defined for all subsets of
a space $\Omega$, but consider invariance with respect to a group $G$ of symmetries that act on a larger space $\Omega^*$. Thus, each member of $G$ is
a bijection of $\Omega^*$ onto itself and the product $gh$ of members of $G$ is the composition of $g$ with $h$. This will let us model such cases as probabilities on
the interval $[0,1]$ that are invariant under translations by letting $\Omega$ be $[0,1]$ and $\Omega^*$ be the
real line $\R$, even though a translation can move a point in $[0,1]$ outside of that interval. As we will see in
Section~\ref{sec:partial}, this approach of enlarging  $\Omega$ is equivalent to talking about invariance under what are
called ``partial group actions'', but makes for a more intuitive presentation.

In the case of full conditional probabilities (or Popper Functions), the literature
contains the ingredients for a complete answer in the case of a strong invariance assumption.
For the sake of completeness we will put together the pieces of this answer in Section~\ref{sec:Popper}, but the main purpose
of this paper is to prove analogous answers for the qualitative and hyperreal cases.

The primary result of this paper will be given in Section~\ref{sec:strong}, and it will be a complete
characterization of when there exists a regular hyperreal measure or a regular strongly
invariant qualitative probability on the powerset of a space $\Omega$ whose superset $\Omega^*$ is
acted on by a group $G$.
Somewhat surprisingly in light of the fact that there are qualitative probabilities that do not arise
from hyperreal ones, it will turn out that the conditions for hyperreal measures and qualitative probabilities
are the same.

In Section~\ref{sec:weak}, we will consider a weaker notion of invariance and argue that it fails to do justice
to our symmetry intuitions. We will end with a speculative discussion of some philosophical issues.

This paper will assume the Axiom of Choice unless explicitly otherwise noted.

\section{Full conditional probabilities}\label{sec:Popper}
A full conditional probability on a space $\Omega$ is a function $P$ from pairs $(A,B)$ of subsets of
$\Omega$ with $B\ne\varnothing$ (the Popper function literature also allows for $B=\varnothing$ with
the trivializing definition $P(A\mid\varnothing)=1$ for all $A$)  from an algebra $\scr F$
to $\R$ satisfying these axioms:
\begin{enumerate}
\item[(C1)] $P(\cdot \mid B)$ is a finitely additive probability function
\item[(C2)] $P(A\cap B\mid C) = P(A\mid C)P(B\mid A\cap C)$.\footnote{The original
version of this paper also included the condition that if $P(A\mid B)=P(B\mid A)=1$, then $P(C\mid A)=P(C\mid B)$, but that follows
from (C1) and (C2).}
\end{enumerate}

Although if we define $P(A)=P(A\mid \Omega)$, the unconditional probability function $P(\cdot)$ in many cases of application will
still fail to be regular, full conditional probabilities can be seen as solving the problem of regularity in two ways.
First, the main
difficulty with lack of regularity is the difficulty in conditionalizing on possible (i.e., non-empty) null-probability events. Full
conditional probabilities allow one to conditionalize on any event other than $\varnothing$. Second, intuitively,
a possible event is more likely than $\varnothing$. Full conditional probabilities capture this intuition by allowing
one to compare the probability of events $A$ and $B$ not by simply comparing $P(A)$ against $P(B)$, but by ``zooming in'' to
the relevant area of probability space and comparing $P(A\mid A\cup B)$ against $P(B\mid A\cup B)$.

The function $P$ is \textit{strongly invariant} with respect to a group $G$ of symmetries on $\Omega^*\supseteq\Omega$
provided that $P(gA\mid B) = P(A\mid B)$ whenever $g\in G$ while $A$ and $gA$ are both subsets of $B$ and in $\scr F$.

Say that two subsets $A$ and $B$ in an algebra $\scr F$ of subsets of $\Omega$ are \textit{$G$-equidecomposable} with respect
to $\scr F$ provided that there is a finite partition $A_1,...,A_n$ of $A$ (i.e., the $A_i$ are pairwise disjoint and their union is $A$) and a
sequence $g_1,...,g_n$ in $G$ such that $g_1 A_1,...,g_n A_n$ is a partition of $B$, and where $A_1,...,A_n,g_1 A_1,...,g_n A_n$ are in
$\scr F$.
If $\scr F$ is the powerset of $\Omega$, we will drop ``with respect to $\scr F$''.

Say that a subset $E$ of $\Omega$ is \textit{$G$-paradoxical} provided it has disjoint subsets $A$ and $B$
each of which is $G$-equidecomposable with $E$ (with respect to the powerset algebra). The famous Banach-Tarski paradox says that if $\Omega$
is three-dimensional Euclidean space and $G$ is all rigid motions, then any solid ball can be partitioned
into two subsets, each of which is $G$-equidecomposable with a solid ball of the same radius as the original.
Thus, any solid ball is $G$-paradoxical.

The following follows by connecting the dots among known results and methods.

\begin{thm}\label{th:full-cond}
    Let $G$ act on $\Omega^*\supseteq \Omega$. The following are equivalent, where all the probabilities are defined with respect to the powerset algebra:
    \begin{itemize}
    \item[(i)] There is a strong $G$-invariant full conditional probability on $\Omega$
    \item[(ii)] $\Omega$ has no nonempty $G$-paradoxical subset
    \item[(iii)] For every nonempty subset $E$ of $\Omega$, there is a $G$-invariant finitely additive real-valued unconditional probability on $E$
    \item[(iv)] For every countable nonempty subset $E$ of $\Omega$, there is a $G$-invariant finitely additive real-valued unconditional probability on $E$
    \item[(v)] $\Omega$ has no nonempty countable $G$-paradoxical subset.
    \end{itemize}
\end{thm}

\textbf{Remark:} The Axiom of Choice is not needed for $\text{(i)}\rightarrow\text{(iii)}\rightarrow{\text{(iv)}}\rightarrow\text{(v)}\rightarrow\text{(ii)}$.

Clearly, (i) implies (iii) and (iv), since $P(\cdot \mid  E)$ will be a $G$-invariant finitely additive real-valued probability on $E$. But a full
conditional probability needs to satisfy the additional coherence constraint (C2) besides defining a probability on every subset. It is interesting
to note that there is no additional difficulty about finding a $G$-invariant conditional probability that satisfies these conditions.
This result may thus be somewhat relevant to those who wish to define weaker conditional probability concepts without conditions like (C2)
(for one approach like this, see \cite{Meehan}).
Moreover, (iv) and (v) show that the root of the difficulty lies with countable subsets, which are classically always measurable.

A special case where there are no $G$-paradoxical subsets of $\Omega$ is when $G$ is supramenable, i.e., does not
itself have $G$-paradoxical subsets~\cite[Section~14.1]{WT16} (when we consider $G$ as
acting on $G$ by left-multiplication). Every abelian (i.e., commutative) group is
known to be supramenable~\cite[Theorem~14.4]{WT16}, so there are many examples where there are strong $G$-invariant
full conditional probabilities. Moreover, if $G$ is supramenable, then no nonempty subset of a space $X$ acted on by $G$
is paradoxical~\cite[p.~271]{WT16}. Thus, precisely for supramenable $G$, every case of $G$-action allows
for $G$-invariant full conditional probabilities.

To prove the theorem, define a \textit{coherent exchange rate} $c$ to be a function with values in $[0,\infty]$ on
pairs $(A,B)$ of subsets of $\Omega$ with $B$ not empty such that:
\begin{enumerate}
\item[(E1)] $c(\cdot \mid B)$ is finitely additive and non-negative
\item[(E2)] $c(A,B)c(B,C)=c(A,C)$ as long as $B$ and $C$ are nonempty and $c(A,B)c(B,C)$ is
well-defined
\item[(E3)] $c(B,B)=1$. (Cf. \cite{Armstrong89,ArmstrongSudderth89,PrussPopper})
\end{enumerate}

Here, we understand that $0\cdot \infty$ and $\infty\cdot 0$ are the undefined cases of multiplication
on $[0,\infty]$. Any full conditional probability $P$ defines a coherent exchange rate
$$
    c_P(A,B)=\frac{P(A\mid A\cup B)}{P(B\mid A\cup B)}
$$
(where $x/0=\infty$ if $x\ne 0$ and $0/0$ is undefined) and any coherent exchange rate $c$
defines a full conditional probability
$$
    P_c(A\mid B) = c(A\cap B,B)
$$
(cf. \cite{ArmstrongSudderth89,PrussPopper}).
Moreover, $P_{c_P}=P$ and $c_{P_c}=c$. We say that $c$ and $P$
correspond to each other provided that $P=P_c$, or, equivalently, $c=c_P$.

We say that $c$ is strongly $G$-invariant provided that we have $c(gA,B)=c(A,B)$ whenever $A,B$ are subsets of $\Omega$, $B$ is nonempty,
and the symmetry $g\in G$ is such that $gA\subseteq\Omega$.

\begin{lem}\label{lem:cP-inv}
    If $c$ and $P$ correspond to each other, then each is strongly $G$-invariant if and only if the
    other is.
\end{lem}

\begin{proof}[Proof of Lemma~\ref{lem:cP-inv}]
    Suppose $P$ is strongly $G$-invariant. Then for any nonempty $A\subseteq\Omega$ such
    that $gA\subseteq\Omega$, we have
    $c(gA,A)=P(A\mid A\cup gA)/P(gA\mid A\cup gA)$ and the ratio is well-defined.
    But $P(gA\mid A\cup gA)=P(A\mid A\cup gA)$ by strong invariance, so the ratio
    is equal to one. Now, fix $A$ and $B$ with $B\ne\varnothing$, and again suppose that $gA\subseteq\Omega$.
    If $A$ is
    empty then $c(gA,B)=0=c(A,B)$. Suppose $A$ is nonempty. Then
    $c(gA,B)=c(gA,A)c(A,B)$ provided the right-hand side is defined. But since
    $c(gA,A)=1$, it must be defined, and indeed is equal to $c(A,B)$. Thus,
    $c(gA,B)=c(A,B)$ and we have $G$-invariance.

    Conversely, suppose $c$ is strongly $G$-invariant and $B$ is nonempty with
    $A\cup gA\subseteq B$. Then $P(A\mid B)=c(A,B)=c(gA,B)=P(gA\mid B)$ and so
    $P$ is also strongly $G$-invariant.
\end{proof}

Say that a $[0,\infty]$-valued finitely additive measure $\mu$ on an algebra $\scr F$ of subsets of $\Omega$ is
$G$-invariant if and only if $\mu(A)=\mu(gA)$ whenever both $A$ and $gA$ are in $\scr F$ and $g\in G$.

\begin{lem}\label{lem:finite}  Let $G$ act on $\Omega^*\supseteq \Omega$. Suppose that for every nonempty subset $E$ of $\Omega$, there is
a $G$-invariant finitely additive measure $\mu:\Power \Omega\to [0,\infty]$ with $\mu(E)=1$. Let $\scr F$ be a finite algebra on $\Omega$.
Then there is a $G$-invariant full conditional probability on $\scr F$.
\end{lem}

\begin{proof}
All the measures in the proof will be finitely additive.
If $\mu$ and $\nu$ are measures on the same algebra, say that $\mu \prec \nu$ provided that
for all $A\in\scr F$, if $\nu(A)>0$, then $\mu(A)=\infty$.
Say that a measure $\mu$ is non-degenerate provided that $0<\mu(A)<\infty$ for some $A$.
Then $\prec$ is known as the R\'enyi order~\cite{Armstrong89,Renyi} and is a strict partial order on non-degenerate measures.

Choose a $G$-invariant probability measure $\mu_1$ on $\scr F$ (there is one on $\Power\Omega$, so restrict
it to $\scr F$).

For $n\ge 1$, supposing we have chosen a $G$-invariant measure $\mu_n$ on $\Power\Omega$, let
$$
   E_{n+1} = \bigcup \{ B \in \scr F : \mu_n(B) = 0 \}.
$$
Note that $\mu_n(E_{n+1})=0$ since $\scr F$ is finite, so $E_{n+1}$ is the largest $\mu_n$-null member of $\scr F$.
If $E_{n+1}=\varnothing$, let $N=n$, and our construction of $\mu_1,\dots,\mu_N$ is complete.

If $E_{n+1}$ is nonempty, choose a additive $G$-invariant measure $\nu$ on $\Power \Omega$ with $\nu(E_{n+1})=1$.
For $A\in\scr F$, let $\mu_{n+1}(A) = \nu(A)$ if $A\subseteq E_{n+1}$ and $\mu_{n+1}(A) = \infty$ otherwise.

I claim that $\mu_{n+1}$ is a $G$-invariant measure on $\scr F$. To check finite additivity, suppose
$A$ and $B$ are disjoint members of $\scr F$. Then if $A$ or $B$ fails to be a subset of $E_{n+1}$, so does $A\cup B$, and so
$\mu_{n+1}(A)+\mu_{n+1}(B)=\infty=\mu_{n+1}(A\cup B)$, and if $A\cup B$ fails to be a subset of $E_{n+1}$, so does at
least one of $A$ and $B$. But if $A$, $B$ and $A\cup B$ are all subsets of $E_{n+1}$, then $\mu_{n+1}$ agrees with $\nu$
as applied to these sets, and $\nu$ is finitely additive.

It remains to check $G$-invariance. Suppose that $A,gA\in\scr F$. If both $A$ and $gA$ are subsets of $E_{n+1}$, the identity
$\mu_{n+1}(A)=\mu_{n+1}(gA)$ follows from the $G$-invariance of $\nu$. If neither is a subset of $E_{n+1}$, then
$\mu_{n+1}(A)=\infty=\mu_{n+1}(gA)$. It remains to consider the case where one of $A$ and $gA$ is a subset of $E_{n+1}$ and
the other is not. Without loss of generality, suppose that $A$ is a subset of $E_{n+1}$ and $gA$ is not (in the
other case, let $A'=gA$ and $g'=g^{-1}$, so $A'$ is a subset of $E_{n+1}$ and $g'A'$ is not). Since $A\subseteq E_{n+1}$,
we have $\mu_n(A)=0$. By $G$-invariance, $\mu_n(gA)=0$, and so $gA\subseteq E_{n+1}$, and thus the case is impossible.

Next note that that $\mu_{n+1} \prec \mu_n$. For if $\mu_n(A)>0$, then $A$ is not a subset of $E_{n+1}$ and so $\mu_{n+1}(A)=\infty$.

The finiteness of $\scr F$ guarantees that the construction must terminate in a finite number $N$ of steps, since
we cannot have an infinite sequence of non-degenerate measures on a finite algebra $\scr F$ that are totally ordered by $\prec$.

We have thus constructed a sequence of $G$-invariant measures $\mu_1,...,\mu_N$ such that $\mu_N\prec \dots\prec \mu_1$.
I claim that for any nonempty $A\in\scr F$, there is a unique $n=n_A$ such that $0<\mu_n(A)<\infty$. Uniqueness follows
immediately from the ordering $\mu_N\prec \dots\prec \mu_1$, so only existence needs to be shown.
By our construction, the only $\mu_N$-null set is $\varnothing$, so $\mu_N(A)>0$. Let $n$ be the smallest index such that
$\mu_n(A)>0$. If $\mu_n(A)<\infty$, we are done. So suppose $\mu_n(A)=\infty$. We cannot have $n=1$, since $\mu_1$ is
a probability measure on $\scr F$. Thus, $n>1$. By minimality of $n$, we must have $\mu_{n-1}(A)=0$. Thus, $A\subseteq E_n$,
and so $\mu_n(A)\le \mu_n(E_n)=1$, a contradiction.

Now, for any $(A,B)\in \scr F\times (\scr F-\{\varnothing\})$, let $P(A\mid B)=\mu_{n(B)}(A\cap B)/\mu_{n(B)}(B)$.
Then $P(\cdot\mid B)$ is finitely additive since $\mu_{n(B)}$ is.

Next, suppose we have $A$, $B$ and $C$ with $A\cap C$ nonempty. If $n(A\cap C)=n(C)$, then let $\mu=\mu_{n(C)}=\mu_{n(A\cap C)}$, so we have
\begin{align*}
    P(A\mid C)P(B\mid A\cap C) &= \frac{\mu(A\cap C)}{\mu(C)} \cdot \frac{\mu(B\cap A\cap C)}{\mu(A\cap C)} \\
        &= \frac{\mu(A\cap B\cap C)}{\mu(C)} = P(A\cap B\mid C).
\end{align*}
Now suppose that $n(A\cap C)\ne n(C)$ so $\mu_{n(C)}(A\cap C)\notin (0,\infty)$. Since $\mu_{n(C)}(A\cap C)\le \mu_{n(C)}(C) <\infty$,
we must have $\mu_{n(C)}(A\cap C)=0$. But then $P(A\mid C)=\mu_{n(C)}(A\cap C)/\mu_{n(C)}(C)=0$ and
$P(A\cap B\mid C)=\mu_{n(C)}(A\cap B\cap C)/\mu_{n(C)}(C)=0$, and so both sides of (C2) are zero.

Finally, $G$-invariance of $P$ follows immediately from $G$-invariance of the $\mu_n$.
\end{proof}

\begin{lem}\label{lem:complete}
Let $G$ act on $\Omega^*\supseteq \Omega$. There is a $G$-invariant
full conditional probability on $\Power\Omega$ if and only if for every nonempty subset $E$ of $\Omega$ there is
a $G$-invariant finitely additive measure $\mu:\Power \Omega\to [0,\infty]$ with $\mu(E)=1$.
\end{lem}

\begin{proof}
First suppose there is a $G$-invariant full conditional probability $P$ on $\Power\Omega$. Then if $E$ were a
nonempty paradoxical subset of $\Omega$, we could partition $E$ into disjoint subsets $A$ and $B$ that could
be  under the action of $G$ to form all of $E$, so that $1=P(E\mid E)=P(A\mid E)+P(B\mid E)=P(E\mid E)+P(E\mid E)=2$
by the finite additivity and $G$-invariance of $P(\cdot \mid E)$. But if $E$ is not a paradoxical subset, then
by Tarski's Theorem~\cite[Cor~11.2]{WT16} there is a $G$-invariant finitely additive measure $\mu$ on $\Power\Omega^*$ which
assigns measure $1$ to $E$, and we can then restrict $\mu$ to $\Power\Omega$.

Conversely, suppose for every $E$ there is a $\mu$ as in the statement of the Lemma.
For a finite algebra $\scr F$ on $\Omega$, let $P$ be a $G$-invariant full conditional probability on $\scr F$
by Lemma~\ref{lem:finite}.
Let $P_{\scr F}(A\mid B) = P(A\mid B)$ for $(A,B) \in \scr F\times (\scr F-\{\varnothing\})$ and $P_{\scr F}(A\mid B)=0$ for
all other $(A,B)\in\Power\Omega\times (\Power\Omega-\{\varnothing\})$. The set $F$ of all finite algebras $\scr F$ on $\Omega$, ordered
by inclusion, is a directed set. Since $[0,1]^{\Power\Omega\times (\Power\Omega-\{\varnothing\})}$
is a compact set by the Tychonoff Theorem, there will be a convergent subnet of the net $(P_{\scr F})_{\scr F \in F}$, and the
limit of that subnet then satisfies the conditions for a $G$-invariant full conditional probability.
\end{proof}

\begin{proof}[Proof of Theorem~\ref{th:full-cond}]
By Tarski's Theorem~\cite[Cor~11.2]{WT16}, there is a $G$-invariant measure $\mu$
on $\Omega^*$ with $\mu(E)=1$ (or, if we prefer, with $0<\mu(E)<\infty$) if and only if $E$ is not $G$-paradoxical.
Such a measure can then be restricted to $\Omega$, and so by Lemma~\ref{lem:complete}, if there are no $G$-paradoxical subsets of $\Omega$, we have
a $G$-invariant full conditional probability.

Conversely, suppose $P$ is a $G$-invariant full conditional probability. And suppose that $E$ is
a nonempty paradoxical set. Let $A$ and $B$ be the two disjoint subsets in the definition of
paradoxicality. Then $P(A\mid E)=P(E\mid E)=1$ and $P(B\mid E)=P(E\mid E)=1$,
which contradicts finite additivity of $P(\cdot\mid E)$.

Thus, (i) and (ii) are equivalent. Moreover, (i) trivially implies (iii) which also trivially implies (iv). Further, (iv) implies
(v) by the same argument as above with $P(\cdot\mid E)$ replaced by the invariant finitely additive measure $\mu$ that assigns $1$ to $E$.

We now show that not-(ii) implies not-(v) and not-(iv). First note that if there is a nonempty paradoxical subset of $\Omega$,
there is a countable nonempty paradoxical subset of $\Omega$. For suppose that $E$ is a nonempty set equidecomposable with two disjoint subsets
$A$ and $B$. The decomposition uses some finite set $S$ of elements of $G$. Let $G_1$ be the subgroup of $G$ generated
by $S$, i.e., the set of all finite products of elements of $S$ and of their inverses. Then $G_1$ is countable. Choose
any $x \in E$. Let $E_1 = G_1 x \cap E$, where $G_1 x = \{ gx : g \in G_1 \}$ is the $G_1$-orbit of $x$. Let $A_1 = G_1 x \cap A$ and
$B_1 = G_1 x \cap B$. Then it is easy to see that $E_1$ and $A_1$ are $G_1$-equidecomposable (just intersect every set involved in the
decompositions $E$ with $G_1 x$) and that so are $E_1$ and $B_1$.

Now, if $E_1$ is nonempty, countable and equidecomposable with two disjoint subsets $A_1$ and $B_1$, there cannot be a $G$-invariant
probability $P$ on $E_1$, since then we would have $1=P(E_1)=P(A_1)=P(B_1)$, which would violate finite additivity. Hence, we have not-(iv), as desired.
\end{proof}

There is also a concept of weak invariance of conditional probabilities: $P(A\mid B)=P(gA\mid gB)$~\cite{PrussPopper}.
It is known that in general weak invariance does not entail strong invariance (notwithstanding
the mistaken \cite[Proposition~1.3]{Armstrong89}) though it does entail in the special case where $G$ has no left-orderable
quotient~\cite{PrussInvariance}.
Otherwise, little appears to be known about this concept, though in Section~\ref{sec:weak} we will show that this concept does not
capture our symmetry intuitions.

\section{Invariant hyperreal measures and strongly invariant qualitative probabilities}\label{sec:strong}
A \textit{partial qualitative probability} $\lessapprox$ on a Boolean algebra $\scr F$ of subsets of $\Omega \subseteq \Omega^*$ is
a relation that satisfies these conditions:
\begin{enumerate}
\item[(Q1)] preorder: reflexivity and transitivity
\item[(Q2)] non-negativity: $\varnothing \lessapprox A$ for all $A$
\item[(Q3)] (finite) additivity: if $A\cap C=B\cap C=\varnothing$, then $A\lessapprox B$ if and only if $A\cup C \lessapprox B\cup C$.
\end{enumerate}
Note that additivity is equivalent to saying that $A\lessapprox B$ if and only if $A-B\lessapprox B-A$.
A total qualitative probability additionally satisfies the totality condition that $A\lessapprox B$ or $B\lessapprox A$ for all $A$ and $B$.
For a good discussion of basic results, see, e.g., \cite{KLST71}.

In a general, a \textit{preorder} is a relation satisfying (Q1) and a preorder $\lessapprox$ is \textit{total} provided that $A\lessapprox B$ or $B\lessapprox A$
for all $a$ and $B$. We will write $A<B$ provided that $A\lessapprox B$ but not $B\lessapprox A$, and $A\approx B$ provided $A\lessapprox B$ and $B\lessapprox A$.

\textit{Regularity} is the condition that $\varnothing < A$ whenever $A\ne\varnothing$.
Finite additivity and regularity together imply that if $A\subset B$, then $A<B$.

\textit{Strong invariance} under a group $G$ acting on $\Omega^*\supseteq\Omega$ then says that $A\approx gA$ whenever $A\in \scr F$
and $g\in G$ are such that $gA\in\scr F$.

\textit{Weak invariance} says that if $A\lessapprox B$, then $gA\lessapprox gB$, as long as both $gA$ and $gB$ are
in $\scr F$.
In Section~\ref{sec:weak}, we shall study weak invariance, and characterize the cases where weak invariance
implies strong invariance. However, we shall show that in some other cases, weak invariance is much weaker
than strong invariance, and does not capture symmetry intuitions well.

A \textit{hyperreal probability} on $\scr F$ is a function from $\scr F$ to some hyperreal field $\hR$ satisfying
the classical axioms of finitely-additive probability:
\begin{enumerate}
\item[(H1)] non-negativity: $P(A)\ge 0$
\item[(H2)] normalization: $P(\Omega)=1$
\item[(H3)] finite additivity: if $A\cap B=\varnothing$, then $P(A\cup B)=P(A)+P(B)$.
\end{enumerate}
Regularity then says that $P(A)>0$ whenever $A\ne\varnothing$ and invariance says that $P(gA)=P(A)$ for all $g$ and $A$.

A hyperreal probability $P$ defines a total qualitative probability $\lessapprox_P$ by specifying that $A\lessapprox_P B$ if and only if
$P(A)\le P(B)$. Regularity for $P$ and $\lessapprox_P$ are equivalent, and invariance for $P$ is equivalent to strong invariance for
$\lessapprox_P$. Interestingly, not every total qualitative probability can be defined in this way~\cite{KPS59}.

An important tool will be local finiteness of action.
Suppose $G$ acts on $\Omega^*\supseteq\Omega$.
Fix $x\in\Omega$ and let $H$ be a subset of $G$.
Let $G_{H,x,0} = \{ x \}$. Given $G_{H,x,n}$, let
$$
    G_{H,x,n+1}=\{ hx : h\in H \text{ and } x \in G_{H,x,n} \text{ and } hx\in\Omega \}.
$$
Let
$$
    G_{H,x} = \bigcup_{n=0}^\infty G_{H,x,n}.
$$
Say that $G$'s action is \textit{locally
finite within $\Omega$} provided that for all $x\in\Omega$ and any finite subset $H$ of $G$, the
set $G_{H,x}$ is finite.

A failure of local finiteness within $\Omega$ means there is a finite $H\subseteq G$ and
a starting point $x\in\Omega$ such that we can visit infinitely many different points starting from $x$, moving by
means of members of $H$ (i.e., moving from some point $y$ to a point $hy$ for
$h\in H$), without ever leaving $\Omega$.

In the special case where $\Omega^*=\Omega$, local finiteness of action
is the same as the concept of the local finiteness of the action of $G$~\cite{Scarparo18}, and if
$\Omega^*=\Omega=G$, it is equivalent to local finiteness of the group $G$, i.e., the claim
that $G$ has no infinite finitely generated subgroups.

The main result of this paper is the following.
Note that the equivalence of (i) and (ii) has been proved by \cite{Scarparo18} in the special case where $\Omega^*=\Omega$
and our proof that (ii) implies (i) will be almost the same.

\begin{thm}\label{th:main} Suppose that $G$ is a group acting on $\Omega^*$. Then the following are equivalent:
    \begin{itemize}
        \item[(i)] The action of $G$ is locally finite within $\Omega$
        \item[(ii)] No subset of $\Omega$ is $G$-equidecomposable with a proper subset of itself
        \item[(iii)] There is a regular $G$-invariant hyperreal probability on $\Power\Omega$
        \item[(iv)] There is a regular total strongly $G$-invariant qualitative probability on $\Power\Omega$
        \item[(v)] There is a regular partial strongly $G$-invariant qualitative probability on $\Power\Omega$
        \item[(vi)] For every $x\in\Omega$, there is a total strongly $G$-invariant qualitative probability on $\Power\Omega$ that
            ranks $\{x\}$ as more probable than $\varnothing$.
        \item[(vii)] No countable subset of $\Omega$ is $G$-equidecomposable with a proper subset of itself
        \item[(viii)] For every countable nonempty subset $E$ of $\Omega$, there is a regular $G$-invariant hyperreal probability on $\Power E$
        \item[(ix)] For every countable nonempty subset $E$ of $\Omega$, there is a regular partial strongly $G$-invariant qualitative probability on $\Power E$.
    \end{itemize}
\end{thm}

\textbf{Remark:} The proofs of $\text{(iii)}\rightarrow\text{(iv)}\rightarrow\text{(v)}\rightarrow\text{(vii)}\rightarrow\text{(ii)}$ do not
 need the Axiom of Choice. The proof of $\text{(ii)}\rightarrow\text{(i)}$ only uses K\"onig's Lemma, so it only needs the Axiom of Countable Choice.

It is interesting to note that (vii)--(ix) show that just as in the case of conditional probabilities,
the difficulty in defining strongly invariant probabilities lies precisely with the countable subsets.

Note that if $G$ is itself a locally finite group, then (i) is automatically satisfied. Thus,
just as supramenable groups $G$ were precisely the groups that had the property that their action (even on
a superset $\Omega^*$)
always admits invariant full conditional probabilities, so too the locally finite groups $G$ are precisely
the ones that have the property that their action (even on a superset $\Omega^*$) always admits invariant
regular hyperreal and total qualitative probabilities.

Since a regular $G$-invariant hyperreal probability $P$ defines a full conditional probability
by the formula $P^*(A\mid B) = \St (P(A\cap B)/P(B))$ (where $\St(x)$ is the standard part of a
finite hyperreal), it is no surprise that condition (ii) entails the
condition that there are no nonempty $G$-paradoxical sets in Theorem~\ref{th:full-cond}. Note that
the condition that there are no nonempty $G$-paradoxical sets is strictly weaker.
For instance, if $G$ is the group of rotations and $\Omega=\Omega^*$ is the unit circle, then $\Omega$ has
no nonempty $G$-paradoxical sets because $G$ is abelian, and so there is a $G$-invariant full
conditional probability. But the set of points $\{x,\rho x,\rho^2 x,\dots\}$ mentioned in the
Introduction is equidecomposable with a proper subset, since a rotation of it by $\rho$
\textit{is} is a proper subset. So, there is no $G$-invariant hyperreal probability or
strongly $G$-invariant qualitative probability on the power set of $\Omega$.

For a slightly more complicated well-known application, suppose that $\Omega$ is the interval $[0,1]$, $\Omega^*=\R$,
and $G$ is the group of all translations on $\R$, which translations we can identify with members of $\R$ acting additively.
Let $r$ be any irrational number in $(0,1/2)$ and let
$H = \{ -1/2, r \}$. Inductively generate a sequence $x_0,x_1,x_2,...$ of numbers in $[0,1]$ by letting $x_0=0$, and then
letting $x_{n+1}=x_n+r$ if $x_n+r \in [0,1]$ and $x_{n+1}=x_n-1/2$ otherwise. It is easy to see that this sequence
has no repetitions, and so $G_{H,0}$ is infinite, so that the translations do not have locally finite action within
$[0,1]$.

Any abelian group all of whose elements have finite order (i.e., $g^n=e$ for some finite $e$) is locally
finite. For if $S$ is a finite subset of elements $a_1,...,a_n$ respectively of finite orders $m_1,...,m_n$, then $S$
generates the finite subgroup of all elements of the form $a_1^{k_1}\cdots a_n^{k_n}$ where $0\le k_i < m_i$.

Thus, an example of a case where the conditions are satisfied is where we have a collection of coin flips,
and our symmetries consist in the reversal of a subset of the results. Here, we can take $\Omega$ and
$G$ to be the group $\Z_2^I$ of zero-one sequences with pointwise addition modulo $2$. Reversal of
a subset $J\subseteq I$ of the results then corresponds to addition of the element of $G$ that is
$1$ for coordinates in $J$ and $0$ for coordinates outside $J$. Every element then has finite order (indeed,
every element other than the identity has order $2$), and so condition (i) is satisfied.

On the other hand, if we have a bidirectionally infinite sequence of coin flips, and our symmetries consist
in translations, then condition (i) is not satisfied. For if $\omega$ is any sequence that has exactly one heads,
and $\tau$ is any non-trivial translation, then $\{ \tau^n \omega : n \ge 0 \}$ is infinite.

\begin{proof}[Proof of Theorem~\ref{th:main}]
We will first see that $\text{(i)}\rightarrow\text{(iii)}\rightarrow\text{(iv)}\rightarrow\text{(v)}\rightarrow\text{(ii)}\rightarrow\text{(i)}$,
then $\text{(iv)}\rightarrow\text{(vi)}\rightarrow\text{(v)}$, and finally see that (vii)--(ix) are equivalent to the earlier
conditions.

Assume (i). Suppose that $H$ is a finite symmetric subset of $G$, i.e., a subset such that if $h\in H$ then $h^{-1}\in H$, and
that $B$ is a finite subset of $\Omega$ with the relative $H$-closure property that if $h\in H$, $b\in B$ and $hb\in\Omega$, then
$hb\in B$. Let $P_{H,B}$ be uniform measure on $B$: $P_{H,B}(U) = \| U\| /\| B\| $ for $U\subseteq B$. This is a regular $H$-invariant probability measure
on $\Power B$. Regularity is trivial. To check for $H$-invariance, suppose $h\in H$ and $hU \subseteq \Omega$. Then by relative
$H$-closure, we have $hU\subseteq B$. But because $h$ is one-to-one on $\Omega^*$, the cardinality of $hU$ must be the same as that of $U$,
so $P_{H,B}(hU)=P_{H,B}(U)$.

Now, let $F$ be the set of all pairs $(H,B)$ where $H$ and $B$ are finite subsets of $G$ and $\Omega$ respectively with $B$ nonempty.
For $U\subseteq \Omega$, let $P_{(H,B)}(U) = P_{H',B'}(U\cap B')$ where $H'=\{ h^{-1} : h \in H \}\cup H$ and
$$
    B'=\bigcup_{x\in B} G_{H',x}.
$$
Note that $B'$ is finite because the action of $G$ is locally finite within $\Omega$.

Say that $(H,B)\preceq (J,C)$ if and only if $H\subseteq J$ and $B\subseteq C$.
Let
$$
    F^*=\{\{\beta \in F:\alpha \preceq \beta\} : \alpha\in F \}.
$$
This is a nonempty set with the finite intersection property. Let $\frak U$ be an ultrafilter
extending $F^*$. Let $\hR$ be an ultraproduct of the reals along $\frak U$, i.e., the set of
equivalence classes $[f]$ of functions $f$ from $F$ to $\R$ under the equivalence relation
defined by saying that $f\sim g$ if and only if $\{ \alpha \in F : f(\alpha)=g(\alpha) \} \in \frak U$.

For $U\subseteq\Omega$, let $P(U)$ be the equivalence class of the function $\alpha\mapsto P_{\alpha}(U)$.
Because each $P_\alpha$ satisfies the axioms of finitely-additive probability, so does $P$. Moreover,
if $U$ is nonempty, then let $\alpha = (\varnothing,\{ \omega\})$ for any fixed $\omega\in U$.
Then, if $\alpha\preceq\beta$, we have $P_\beta[U] > 0$. Thus, $\{ \beta \in F : P_\beta[U] > 0 \}$ contains
$\{ \beta\in F : \alpha\preceq\beta \}$, which is a member of $F^*$. Since $\frak U$ is an ultrafilter extending $F^*$,
it follows that $\{ \beta \in F : P_\beta[U] > 0 \} \in \frak U$, so $[\beta \mapsto P_{\beta}[U]] > [0]$, where $[0]$ (often just written
``$0$'' for convenience) is the equivalence
class of the function that is identically zero on $F$. Hence, we have regularity.

It remains to show that we have $G$-invariance. Suppose $g\in G$ and $U\subseteq\Omega$ is such that $gU\subseteq\Omega$.
We must show that $P(U)=P(gU)$. This is trivial if $U$ is empty, so suppose that $U$ is nonempty. Let
$\alpha = (\{g\}, \{\omega\})$ for some $\omega\in U$. As in the case of regularity, all we need to show is that
$P_\beta(gU) = P_\beta(U)$ if $\alpha\preceq\beta$. Suppose then that $\beta=(H,B)$. Replacing $H$ and $B$ with $H'$ and $B'$
as per their earlier definitions if necessary, we may suppose that $H$ is symmetric and $B$ has the relative $H$-closure
condition that if $b\in B$, $h\in H$ and $hb\in\Omega$, then $hb\in B$. Let $U'=U\cap B$. Then $gU'=gU\cap B$ by
the relative $H$-closure of $B$. Then $P_\beta(U)=P_{H,B}(U')=P_{H,B}(gU')=P_\beta(gU)$, as desired.

The implication (iii)$\rightarrow$(iv) follows by letting the hyperreal probability
define the qualitative probability, and (iv)$\rightarrow$(v) is trivial.

Assume (v). The axioms of qualitative probability imply that if $A_1,...,A_n$ are disjoint and $B_1,...,B_n$ are also disjoint,
and $A_i\approx B_i$ for all $i$, then $\bigcup_{i=1}^n A_i \approx \bigcup_{i=1}^n B_i$ by \cite[Lemma~5.3.1.2]{KLST71}.
It follows that if we have strong $G$-invariance, then any two $G$-equidecomposable sets are equally
likely. Suppose then $A$ is $G$-equidecomposable with a subset $B$ of itself. Then $B\cup \varnothing =B\approx A=B\cup (A-B)$.
So, by additivity $\varnothing\approx A-B$, which by regularity can only be true if $A-B$ is empty,
i.e., $B$ is an improper subset of $A$. So, given (v), no subset of $\Omega$ can be $G$-equidecomposable
with a proper subset of itself. That yields (ii).

Now we need to show that (ii) implies (i). For a contrapositive proof that is based on ideas of \cite{Scarparo18}, suppose the action of $G$ is not locally finite
within $\Omega$. Thus, $G_{H,x}$ is infinite for some finite $H\subseteq G$ and $x\in G$. Without loss of generality
suppose that $H$ contains the identity $e$ and is symmetric. Following \cite{Scarparo18}, we can consider $G_{H,x}$ an
infinite connected graph
where there is an edge between $x,x'\in G_{H,x}$ if $x'=hx$ for some $h\in H-\{e\}$. By K\"onig's Lemma
there is a ray on this graph, i.e., an infinite path with a starting point and no repetitions.
Suppose $x_1,x_2,\dots$ is a ray, and suppose that $x_{n+1}=h_n x_n$ for $h_n \in H-\{e\}$.
Let $r = \{ x_n : n \ge 1 \}$.

For $h\in H$, let $A_h = \{ x_n : n\ge 1 \text{ and }h_{n} = h  \}$. Note that $(A_h)_{h\in H}$ is a finite partition of $r$
and $(hA_h)_{h\in H}$ is a finite partition of $\{ x_n : n > 1 \}\subset r$. It follows that $r$ is $H$-equidecomposable with a
proper subset, and since $H\subseteq G$, the proof is complete.

Next, clearly (iv) implies (vi).

Suppose (vi) is true. Let $\{ \lessapprox_i : i \in I \}$ be the set of
all partial strongly $G$-invariant qualitative probabilities on $\Power\Omega$, indexed with some set $I$.
Let $\prec$ be a strict well-ordering of $I$. Define the lexicographic ordering $A\lessapprox B$ if and
only if either $A\approx_i B$ for all $i$ or else $A<_i B$ for the $\prec$-first $i$ for which $A\not\approx_i B$.
This is easily checked to be a $G$-invariant total preorder.
Non-negativity and additivity for $\lessapprox$ follows from the corresponding conditions for the $\lessapprox_i$.
Moreover, it is regular, since $\varnothing\lessapprox_i A$ for all
$A$, but if $A$ is nonempty and hence contains some $x$, then for some $i$ we have $\varnothing<_i \{ x \}\lessapprox_i A$ by (iv).

Thus, (i)--(vi) are equivalent. Condition (ii) implies (vii) trivially. The equivalence of (vii), (viii) and (ix)
follows from the equivalence of (ii), (iii) and (v) as applied to a nonempty countable $E$ in place of $\Omega$.

It remains to show that not-(ii) implies not-(viii). This is much as in the proof of Theorem~\ref{th:full-cond}. If $E$ is equidecomposable with a proper
subset $A$, let $G_1$ be the countable subgroup of $G$ generated by all the elements involved in the equidecomposition.
Fix $x \in E-A$. Let $E_1 = E \cap G_1 x$, set $A_1 = A_1 \cap G_1 x$, and note that $E_1$ will be $G_1$-equidecomposable
with $A_1$, which is a proper subset of $E_1$ as $x\in E_1-A_1$.
\end{proof}

\section{Weak invariance}\label{sec:weak}
For both full conditional probabilities and qualitative probabilities, we have a concept of weak
invariance. In both cases, however, I will argue that this concept fails to capture intuitive
symmetries in fair lotteries.

Recall that a qualitative probability $\lessapprox$ is weakly $G$-invariant on $\Omega\subseteq\Omega^*$
provided that $gA\lessapprox gB$ if and only if $A\lessapprox B$, assuming all four sets $A$, $B$, $gA$
and $gB$ are subsets of $\Omega$.
We say that a full conditional probability $P$ is weakly $G$-invariant in these circumstances
provided that $P(gA\mid gB)=P(A\mid B)$ under the same conditions.

For simplicity, in this section I restrict discussion to the case where $\Omega^*=\Omega$. Then under a
certain group-theoretic condition on $G$, weak invariance implies strong invariance.

Suppose that $\scr F$
is a $G$-invariant algebra of subsets of $\Omega$, i.e., if $A \in \scr F$ and $g\in G$, then $gA\in\scr F$.

Say that $\le$ is a \textit{left order} (respectively, \textit{total left preorder}) on a group $G$ provided that $\le$ is a total
order (total left preorder) such that $a\le b$ if and only if $ca\le cb$ for all $a,b,c\in G$.  A preorder is
\textit{non-trivial} provided that for some $a$ and $b$ we have $a<b$. A quotient of a group is \textit{non-trivial} provided
that it contains more than one element. The equivalence of (i) and (iv) in the result below is due to \cite{PrussInvariance}.

\begin{thm}\label{th:wk-strong}
The following are equivalent:
\begin{itemize}
\item[(i)] $G$ has no non-trivial quotient with a left order
\item[(ii)] There is no non-trivial total left preorder on $G$
\item[(iii)] Whenever $G$ acts on a set $\Omega$ with a $G$-invariant algebra $\scr F$, every weakly $G$-invariant
    total qualitative probability on $\scr F$ is strongly $G$-invariant
\item[(iv)] Whenever $G$ acts on a set $\Omega$ with a $G$-invariant algebra $\scr F$, every weakly $G$-invariant
    full conditional probability $P$ on $\scr F$ is strongly $G$-invariant.
\end{itemize}
\end{thm}

\begin{cor}\label{cor:wk-strong}
Suppose that $G$ is a group acting on $\Omega$ and generated by elements of finite order.
If $P$ is a weakly $G$-invariant full conditional probability, then $P$ is strongly
$G$-invariant. If $\lessapprox$ is a weakly $G$-invariant total qualitative probability, then
it is strongly $G$-invariant.
\end{cor}

Here, an element $g$ has finite order $n$ provided that $g^n=e$, the identity element.
A group $G$ is generated by a subset $S$ provided that every non-identity
element of $G$ is a finite product of elements of $S$. If a group is generated by elements of
finite order, then the same is true for every quotient of the group. But a non-trivial group generated by elements
of finite order cannot have a left order. For suppose that $g\ne e$, $g^n=e$ and $\le$ is a left order. Then either $e<g$ or $g<e$. If $e<g$,
then $g<g^2$ and $g^2<g^3$ and so on up to $g^{n-1}<g^n=e$, and so by transitivity $g<e$, a contradiction. If $g<e$, then $e<g^{-1}$, and
we run the previous argument with $g^{-1}$ in place of $g$. So, the corollary follows from condition (i) in the theorem.

Some very natural cases satisfy the
finite-order generating set condition. For instance, all rigid motions on the line, in the plane or on the circle
can be generated by reflections, which have order two. Similarly, the group of reversals of subsets of results in
coin-flipping setups is not only generated by elements of order two, but all non-identity elements have order two.

\begin{proof}[Proof of Theorem~\ref{th:wk-strong}]
If $G/N$ is a non-trivial quotient with a left order $\le$ for a normal subgroup $N$, then define $a \precapprox b$ if and only if $aN \le bN$ for $a,b\in G$.
This is clearly a non-trivial total preorder. Thus, not-(i) implies not-(ii). The converse is due to \cite{Cornulier13}, and a proof
is also given in print in \cite{PrussInvariance} as part of the proof of the main theorem. Thus, (i) and (ii) are equivalent. The equivalence of
(i) and (iv) is shown in \cite{PrussInvariance}.

We now show that not-(i) implies not-(iii). Let $\le$ be a non-trivial total left order on $\Omega=G/N$, and suppose that $G$
acts on $\Omega$ in the canonical way: $g(hN)=(gh)N$.
Let $\scr F$ be the algebra of all finite or co-finite subsets of $\Omega$. Define $A\lessapprox B$ just in case for every $x\in A-B$
there is a $y\in B-A$ such that $x\le y$.

The reflexivity of $\lessapprox$ is trivial as is the positivity condition $\varnothing \lessapprox A$, and the additivity condition is very easy.
Totality is also not hard. Suppose that we don't have $A\lessapprox B$. Then there is an element $x$ of $A-B$ such that
for no $y$ in $B-A$ do we have $x\le y$. By the totality of $\le$, for every $y$ in $B-A$ we must have  $y<x$, and so $B\lessapprox A$.

Transitivity is a bit harder.
Observe that if $A$ is finite and $B$ is co-finite, then $A<B$.  For by non-triviality, there is an element $x$ of $G/N$ such that
$e<x$, where $e$ is the identity element of $G/N$ (for there is an element $y\ne e$, and by totality either $e<y$ in which case we can
let $x=y$ or $y<e$ in which case we can let $x=y^{-1}$). Since $\le$ is a left order, it follows that $x^n<x^{n+1}$ for all $n$, and so by transitivity
$e< x< x^2<\cdots$.
In particular, we learn that $G$ is infinite, and so co-finite sets have infinitely many members.
Let $z$ be the largest element of $A$. Then since $B$ is co-finite, it must contain infinitely many of the elements of the form $x^n z$ for
$n>0$. Choose one such element $x^n z$. Since $e< x^n$, we have $z < x^n z$, so $A<B$ as $z=\max A$.

We need to show that if $A\lessapprox B$ and $B\lessapprox C$, then $A\lessapprox C$. Suppose first that we have shown this for all finite
$A$, $B$ and $C$. Now suppose that exactly one of $A$, $B$ and $C$ is co-finite. Since a co-finite set cannot be $\lessapprox$-smaller than
a finite set, the co-finite set must be $C$, and then we have $A<C$ by what we showed before. Next, suppose exactly two of the sets are
co-finite, and the third is finite. The finite set must then be $A$, and once again we have $A<C$. The remaining case is where all three
sets are co-finite. Let $D=A\cap B\cap C$. This is a co-finite set. Let $A'=A-D$, $B'=B-D$ and $C'=C-D$. These are finite sets, and
$A'\lessapprox B'$ if and only if $A\lessapprox B$, $B'\lessapprox C'$ if and only if $B\lessapprox C$ and $A'\lessapprox C'$ if and only
if $A\lessapprox C$. If we have transitivity for finite sets, we have transitivity for $A$, $B$ and $C$, then.

I now claim that if $A$, $B$ and $C$ are finite and $A\lessapprox B$, $B\lessapprox C$ and $C\lessapprox A$,
then $A=B=C$. To see this, let $D=(A-B)\cup (B-C)\cup (C-A)$.

If $D$ is empty then
$A\subseteq B\subseteq C\subseteq A$, so $A=B=C$. Now suppose $D$ is nonempty.
Let $x$ be the largest element of the finite set $D$. Renaming $A$, $B$ and $C$ if needed, we may suppose that $x\in A-B$.
Then there is an element $y$ of $B-A$ such that $x\le y$ since $A\lessapprox B$. We cannot have $x=y$,
so $x<y$.
This element $y$ cannot be a member of $B-C$, by maximality of $x$ within $D$. Since $y$ is a member of $B$, it
follows that $y$ must also be a member of $C$. But $y$ is not in $A$. Hence $y\in C-A$, which also violates the maximality of $x$.

Now, suppose that $A\lessapprox B$ and $B\lessapprox C$ and all three sets are finite. If we do not have
$A\lessapprox C$, then by totality we have $C\lessapprox A$, and so by what we have just proved we have $A=B=C$,
so $A\lessapprox C$ by reflexivity, a contradiction.

The weak $G$-invariance of $\lessapprox$ follows from the fact that $\le$ is a left preorder on $G$.
We now show that $\lessapprox$ is not strongly $G$-invariant. Choose $a,b$ in $G/N$ with $a<b$. Then $\{ a \}<\{ b \}$.
Write $a = xN$ and $b=yN$ for $x,y\in G$. Let $z=xy^{-1}$, so $zb=a$. Thus, $z \{ b \} = \{ a \} < \{ b \}$, contrary to strong $G$-invariance.

We now show that not-(iii) implies not-(ii), mirroring the analogous proof in \cite{PrussInvariance}. Let $\lessapprox$ be a weakly but not strongly $G$-invariant
total qualitative probability on a $G$-invariant algebra $\scr F$ of subsets of some set $\Omega$.
Then by lack of strong invariance, there is an $A\in\scr F$ and $g\in G$ such that $A\not\approx gA$. Define
$\precapprox$ on $G$ by $x\precapprox y$ if and only if $xA\lessapprox yA$. This is a total preorder
on the set $G$
and
it is non-trivial since $A\not\approx gA$. Moreover, if $h\in G$, and $x\precapprox y$, then $xA\lessapprox yA$,
and so by weak invariance $gxA\lessapprox gyA$, and hence $gx\precapprox gy$, so it is a total left preorder
on the group $G$.
\end{proof}

There are cases where there are weakly invariant qualitative probabilities on a powerset
but no strongly invariant ones. For instance, let $\Omega=G=\mathbb Z$ be the set of integers, acted
on by addition. Then $\mathbb Z$-invariance is translation invariance.
Now $\mathbb Z$ is finitely generated (since it's generated by the element $1$) and infinite,
and hence not locally finite, so by Theorem~\ref{th:main} there is no strongly invariant qualitative
probability on it. However, West~\cite{West20} has proven that there is a regular weakly $\mathbb Z$-invariant qualitative probability
on the powerset $\mathbb Z$ (the proof generalizes to any abelian group).

Nonetheless, the notion of weak invariance does not capture the symmetry notions that invariance is meant
to capture. For it turns out that any regular weakly translation-invariant qualitative probabilities
on the powerset of the integers $\mathbb Z$ exhibit significant skewing.

Let $L_n = \{ m \in \mathbb Z : m < n \}$ and $R_n = \{ m \in \mathbb Z : n \le m \}$ be left- and right-halves
of $\mathbb Z$ split at $n$.

\begin{prp}\label{prp:skew1} Suppose $\lessapprox$ is a regular weakly translation-invariant total qualitative
probability on the powerset of $\mathbb Z$. Then one of the following statements is true:
    \begin{itemize}
        \item[(i)] For every $m$ and $n$, $L_m<R_n$
        \item[(ii)] For every $m$ and $n$, $L_m>R_n$.
    \end{itemize}
\end{prp}

In other words, $\lessapprox$ must either favor all the right halves over all the left halves, or vice versa.
This also shows that if our regular total qualitative probability $\lessapprox$ has weak translation
invariance, it does not have weak reflection invariance for \textit{any} reflection $\rho$ (i.e., with respect to the
group consisting of $\rho$ and the identity). For weak reflection invariance under $\rho$ would imply
strong reflection invariance under $\rho$ by Theorem~\ref{th:wk-strong}, which would violate both (i)
and (ii) in Proposition~\ref{prp:skew1}.

\begin{proof}[Proof of Proposition~\ref{prp:skew1}] By regularity and additivity, we have the strict
monotonicity properties that $L_m<L_n$ and $R_m>R_n$ whenever $m<n$. Suppose (i) is false. Then
$L_m \gtrapprox R_n$ for some $m$ and $n$ by totality.

Suppose first that $m<n$. Then $L_n > L_m \gtrapprox R_n$,
so $L_n > R_n$. By weak translation invariance, it follows that for every $k$ we have $L_k > R_k$
(since $L_k = (k-m)+L_m$ and $R_k = (k-m)+R_m$). Now fix any $j$ and $k$. If $j \ge k$, then $L_k > R_k\gtrapprox R_j$
by monotonicity. If $j<k$, then also by monotonicity $L_k > L_j > R_j$. So, we have (ii).

Now suppose that $m\ge n$. Then $L_m \gtrapprox R_n \gtrapprox R_m$ by monotonicity. By weak invariance (translating to
the left by one), we have $L_{m-1} \gtrapprox R_{m-1}$. But $R_{m-1} > R_m$, so $L_{m-1} > R_m$. Letting
$m' = m-1$ and $n' = m$, we have $L_{m'} \gtrapprox R_{n'}$ and $m'<n'$. By the previous case (where $m<n$) we
get (ii) again.
\end{proof}

In fact, if we want, we can use a modification of West's method to manufacture even more extreme cases of
skewage that are compatible with weak invariance, both for qualitative probabilities and for
full conditional probabilities.

\begin{prp}\label{prp:skew2}
There is a weakly translation-invariant total qualitative probability $\lessapprox$ on the powerset
of $\mathbb Z$ such that $\{ n \} \approx \{ m \}$ for all $n$ and $m$, and such that if
$B$ contains infinitely many positive integers and $A$ only finitely many positive integers, then
$A<B$.
\end{prp}

\begin{prp}\label{prp:skew2p}
There is a weakly translation-invariant full conditional probability $\lessapprox$ on $\mathbb Z$
such that for all $n$ and $m$, $P(\{n\}\mid \{n,m\})=P(\{m\},\{n,m\})$, and such that if
$B$ contains infinitely many positive integers and $A$ only finitely many positive integers, then
$P(A\mid A\cup B)=0$ and $P(B\mid A\cup B)=1$.
\end{prp}

The proofs are given in the Appendix.

The probabilities in these propositions strongly favor infinite sets of positive integers. For instance,
both of them favor the set of positive integers over the set of negative integers. That is no surprise
in the qualitative case: by Proposition~\ref{prp:skew1}, $\lessapprox$ has to favor the positive integers
over the negative integers, or vice versa. But it gets worse. For we can make $B$ be any sparse infinite
set of positive integers, say $\{ 2^{2^n} : n\ge 1 \}$, and let $A$ be all negative integers, and
then the qualitative measure in Proposition~\ref{prp:skew2} still makes it be that $A<B$ and the full conditional
probability of Proposition~\ref{prp:skew2p} makes it be that $P(A\mid A\cup B)=0$ and $P(B\mid A\cup B)=1$.

There is no broadly accepted concept of fairness for an infinite lottery.
That any two individual tickets are equally likely to win is generally taken to be a necessary condition.
But intuitively, a significant degree of symmetry and lack of systematic bias is also called for (\cite{Norton2018}
recommends invariance with respect to \textit{all} permutations, but that is likely too strong).  A lottery with
our radically skewed probabilities that nonetheless treat all individual integers as equiprobable does not intuitively
appear to be fair.  Thus, weak translation invariance plus equiprobability of singletons does not appear to be sufficient to capture
our intuitions of fairness and symmetry. We would probably do better to focus on strong invariance---but we saw that
that's harder to get.

\section{Partial actions}\label{sec:partial}
There is some literature on the partial actions of groups which captures the same
phenomenon as we captured above by letting $G$ act on a superset $\Omega^*$ of $\Omega$. Specifically, a partial
action of a group $G$ on a space $\Omega$ is a collection of one-to-one functions $(\theta_g)_{g\in G}$
defined on subsets (possibly empty) of $\Omega$, such that:
\begin{enumerate}
\item[(PA1)] $\theta_e$ (where $e$ is the
identity in $G$) is the identity function on $\Omega$,
\item[(PA2)] $\Dom \theta_g = \Range \theta_{g^{-1}}$ for all $g$
\item[(PA3)] if $g,h\in G$ and $x\in\Omega$ are such that $x\in\Dom \theta_g$
and $\theta_g(x)\in\Dom \theta_h$, then $x\in\Dom \theta_{hg}$ and $\theta_{hg}(x)=\theta_h(\theta_g(x))$.
    (Cf.\ \cite{Exel})
\end{enumerate}

For a (full) group action of $G$ on a space $\Omega^*$ containing $\Omega$, we can define a
partial action on $\Omega$ by taking $\theta_g$ to be a function on $\Omega^*$ with domain $(g^{-1}\Omega)\cap\Omega$
and defined by $\theta_g(x)=gx$. We say that the partial action $(\theta_g)_{g\in G}$ is then a
restriction of the full action of $G$ on $\Omega^*$.

Every partial action is a restriction of a full
action~\cite[Theorem~1.3.5]{Exel}.
Consequently, our results about actions on a larger containing set $\Omega^*$ are equivalent to results about partial group
actions on $\Omega$ itself.

\section{Philosophical remarks}
\subsection{Some examples}
Let us restrict our attention to the stronger forms of invariance under $G$. Then
it is strictly easier to get invariant full conditional probabilities than to get either invariant regular hyperreal
or qualitative (regardless whether total or partial) probabilities for all subsets of $\Omega$. The hyperreal and qualitative probabilities are equally
hard to get, despite the fact that not every qualitative probability derives from a hyperreal probability.
Our characterizations show that the difficulty in getting invariance under symmetries always has to do with subsets that are
countable, and hence classically measurable.

In the table we have a summary of some examples, some which were already discussed, and most of which
follow quickly from Theorems~\ref{th:full-cond} and \ref{th:main}.

\begin{figure}
\caption{Existence of (strongly) invariant non-classical probabilities.}
\begin{center}
\newcolumntype{L}[1]{>{\raggedright\let\newline\\\arraybackslash\hspace{0pt}}m{#1}}
\begin{tabular}{|L{2.4cm}|L{2.4cm}|L{2.2cm}|L{2.2cm}|L{2.2cm}|}
\hline
\linespread{1}
\textbf{Case} & \textbf{Symmetries} & \textbf{Full conditional} & \textbf{Regular hyperreal} & \textbf{Regular qualitative} \\
\hline
finite space & any & yes & yes & yes \\
\hline
infinite lottery on $\Z$ & translations & yes & no & no \\
\hline
infinite lottery on $\Z$ & reflections & yes & no & no \\
\hline
infinite lottery on $\Z$ & all permutations & no & no & no \\
\hline
infinite lottery on any set & permutations affecting only finite subsets & yes & yes & yes \\
\hline
bidirectional infinite sequence of coin flips & translations & yes & no & no \\
\hline
bidirectional infinite sequence of coin flips & translations and finite subset reversals & no & no & no \\
\hline
arbitrary infinite sequence of coin flips & permutations & no & no & no \\
\hline
arbitrary infinite sequence of coin flips & reversal of subset of results & yes & yes & yes \\
\hline
$[0,1]$ & translations & yes & no & no \\
\hline
circle/spinner & rotations & yes & no & no \\
\hline
surface of sphere & rotations & no & no & no \\
\hline
subset of $\R^n$ containing cube, $n\ge 2$ & rigid motions & no & no & no \\
\hline
$\R^n$, $n\ge 1$ & translations & yes & no & no \\
\hline
$\R^n$, $n\ge 1$ & translations and reflections of coordinates & yes & no & no \\
\hline
\end{tabular}
\end{center}
\end{figure}

In the table, all the ``yes'' entries under ``Full conditional'' are due to the symmetry group being
supramenable. In all but the first and last cases, the lottery reflection case and the lottery with finite subset permutations, this is due to
its being abelian. In the last case the result is due to \cite[Theorem~3]{PrussPopper}, and the group
in the lottery reflection case is just a subgroup of the group in the last case. The case of the lottery with
finite subset permutations is due to Ian Slorach and clearly satisfies local finiteness of the group action.

A sufficient condition to lack any of the three types of invariant probabilities for all subsets is that $\Omega$ has a
$G$-paradoxical subset. The triple ``no'' in the sphere case is due to the Banach-Tarski paradox and in the
set-containing-cube (where a two-dimensional ``cube'' is a square) case follows from the construction of a
bounded paradoxical set in two dimensions by \cite{Just88}.

The triple ``no'' for the infinite lottery case follows from the fact that if $G$ is any countable non-supramenable
group acting on itself (e.g., the free group on two elements~\cite[Theorem~1.2]{WT16}), then the answer will be
a triple ``no''. Then via the bijection between $G$ and $\Z$, we can also take such a group $G$ to act on $\Z$, with each element's
action being a permutation, and we will still have the triple ``no''.

The triple negative result for the infinite
coin toss case under permutations then follows from the fact that if we fix a countably infinite subset $I$ of
the coins, and let the subset $\Omega_0$ be the coin toss results that are tails everywhere except for one heads
result in $I$, then we can embed $\Omega_0$ in the lottery on $I$ (for $\omega\in \Omega_0$, let $\phi(\omega)\in I$ be
the position of the unique heads), and $I$ bijects with $\Z$.

Finally, the triple negative result for the bidirectionally infinite coin toss case with finite subset reversals
is a modification of \cite{Williamson07}. Let $H_n$ ($T_n$) be the event of getting heads (tails) on toss $n$ and let $H_n^+$ be the
event of getting heads on all the tosses $n,n+1,\dots$. Then $E=H_2^+$ can be partitioned into $A=H_1\cap H_2^+$ and $B=T_1\cap H_2^+$.
Moreover, shifting $A$ to the right yields $E$, while reversing the result of toss $1$ in $B$ yields $A$, and shifting it to the
right yields $E$. Thus, $E$ is equidecomposable with both $A$ and $B$, and hence paradoxical, so we have a triple negative row.

All the other entries in the table were either discussed above in the paper, or are easy consequences of results earlier  discussed in this paper.

With the possible exception of the arbitrary infinite sequences of coin flips under permutations, it can be easily checked
that all the ``no'' entries in the above table can be proved without any use of the Axiom of Choice. In the case of
infinite sequence of coin flips under permutations, the Axiom of Choice is only used to show that the infinite set of coins
has a countable subset, which only needs the very plausible Axiom of Countable Choice---and it won't need any Choice if the infinite
set is itself countable.

Hence, denial of the Axiom of Choice does not appear to be a helpful tool to saving symmetries, especially as the proofs of the
existence of the non-classical probabilities typically make use of some version of the Axiom of Choice (see also \cite{Pruss14}).

\subsection{Intuitions}
There are multiple considerations that apply when choosing between probabilistic frameworks. Thus, non-classical approaches
have advantages vis-\`a-vis regularity or being everywhere defined.\footnote{I am grateful to an anonymous reader for this point.}
On the other hand, the classical approach has the advantage of being able to preserve significantly more symmetries.

The table above may seem to suggest that the full conditional probability framework has some advantage over the hyperreal and
qualitative frameworks due there being so many positive answers in the full conditional
column of the table. Rows that have ``no'' in the hyperreal and qualitative columns but ``yes'' in the full conditional column
correspond to cases where there is a subset $E$ of the space that is equidecomposable with a proper subset $E_1$. The reason that
full conditional probabilities can handle such cases is because there is an important sense in which full conditional probabilities
only partly do justice to intuitions about regularity. The standard way to generate a probability comparison with full conditional
probabilities is to say that $A\lessapprox B$ just in case $P(A\mid A\cup B)\le P(B\mid A\cup B)$. It is easy to check that with
this comparison, we have the regularity condition that $\varnothing<A$ for every nonempty $A$. But we need not have the stronger
regularity condition that if $A\subset B$ then $A<B$, and that is what allows full conditional probabilities to handle a set $E$
that is equidecomposable with a proper subset $E_1$. In such a case, invariance will ensure that $P(E_1\mid E_1\cup E)=1=P(E\mid E_1\cup E)$,
so $E_1\approx E$, contrary to the stronger regularity condition. (Note that $\lessapprox$ in this case will not satisfy the additivity
condition for qualitative probabilities.)

One philosophically interesting thing to note about the table is that in all of the cases, we seem to be able to imagine
probabilistic situations where the requisite symmetries are intuitively correct, and the cases where non-classical
probabilities of some given type cannot allow for these symmetries do not appear to differ in any significant way with respect
to that intuition from some of the cases where the symmetries can be had.

For instance, one of the two infinite cases in the table where regular qualitative/hyperreal probabilities can be defined invariantly is
the infinite collection of
coin flips, with symmetries being reversals of a selected subset. But while invariance under such symmetries is intuitive, it
is just as intuitive that in a bidirectionally infinite ``horizontal'' sequence of coin flips $...,X_{-2},X_{-1},X_0,X_1,X_2,...$,
our probabilities should be invariant under horizontal shifts of results, and yet qualitative/hyperreal probabilities cannot
be invariantly defined then.

Likewise, we can have full conditional probabilities for a bidirectionally infinite sequence of coin tosses invariant under
reversal of any finite subset of results (or even an infinite one), and we can have ones that are invariant under any translation,
but not under both translations and reversals, as we saw above.

Similarly, there seems to be little difference between throwing a dart at random at the interval $[0,1]$ and expecting translational
symmetry and throwing a dart at random at $[0,1]^2$ and expecting translational and rotational symmetry, while in the former case
we have full conditional probabilities and in the latter we do not.

And there is nothing particularly special about translations in the case of an infinite sequence of coin tosses: any permutation
of the coins should just as intuitively preserve probabilities as a translation, and yet for translations we have
``yes'' for full conditional probabilities and for general permutations a triple negative.

In any case, it seems that regardless of whether we prefer full conditional, regular hyperreal or regular qualitative probabilities,
we need to abandon our symmetry intuitions in some but not other cases in ways that may seem intuitively \textit{ad hoc}. This
is an advantage for the classical framework of real-valued probabilities, where not all subsets have defined probabilities and where
we lose regularity, but at least it is much easier to get symmetries. Lebesgue measure on Euclidean space is always translation-
and rotation-invariant, and product measures for
infinite coin-flip situations will be invariant under all shifts, and indeed under all permutations of the coins.

It is also worth noting that the sets that ``block'' the existence of strongly invariant probabilities can always be taken to be countable,
and hence are all going to be measurable in the context of classical probabilities, though that measure may be zero. We can either say that
the greater resolving power of the non-classical approaches brings to light difficulties with these sets that classical probability ignores,
or we might think that the classical approach is superior in allowing for the symmetries.

\subsection{Closing remarks}
But this is philosophical speculation, and perhaps readers will find dissimilarities between intuitions about symmetry
of probabilities that align with the necessary and sufficient conditions given by the theorems of this paper. A further
area for future research is to find ways to measure the degree of deviation from symmetry and see whether deviations of
that degree are acceptable.\footnote{I am grateful to an anonymous reader for this suggestion.}

In any case, the main point of this paper is to provide the technical
characterizations to help inform such philosophical discussion.

Finally, there is a need for more mathematical investigation of the weaker forms of invariance. However, the results of Section~\ref{sec:weak} suggest that
weaker forms of invariance are insufficient to do justice to our intuitions about symmetry.\footnote{I am grateful to Alexander Meehan and Ian Slorach
for encouragement and discussion, to the participants of the Princeton/Rutgers foundations of probability working group for their many interesting comments,
and to two anonymous readers for a number of comments that have significantly improved this paper.}

\section*{Appendix: Construction of highly skewed weakly invariant probabilities}
To prove Propositions~\ref{prp:skew2} and \ref{prp:skew2p}, we use the methods of \cite{West20}.
Let $\scr B$ be the
(real) vector space of all bounded functions from $\Z$ to $\mathbb R$ (i.e., functions $f$ such that there is a real $M$ such that for all $x$ we have
$|f(x)| <M$).
Let $\scr M$ be the subset of $\scr B$ consisting of non-negative functions that are strictly positive at at least one point of $\Z$ and that are finitely
supported, i.e., are zero except at finitely many points. For two functions $f$
and $g$ in $\scr B$, define the convolution $f*g$ by:
$$
    (f*g)(x) = \sum_{y=-\infty}^\infty f(y)g(x-y),
$$
whenever this sum is defined. The convolution will always be defined and a member of $\scr B$ when one
of the functions is in $\scr B$ and the other is finitely supported. It is easy to check that convolution
is commutative on $\scr M$ (this uses the commutativity of $(\Z,+)$), and that
we have the associativity property $a*(\phi*\psi)=(a*\phi)*\psi$
whenever $a\in \scr B$ and $\phi,\psi\in \scr M$. Observe that if $\delta_x$ is the function on $\Z$ that is
zero except at $x\in \Z$ where it is equal to one, then $f*\delta_0 = f$.

Define the relation $\sim$ on $\scr B$ by $a\sim b$ if and only if $a * \phi = b * \psi$ for
some $\phi$ and $\psi$ in $\scr B$. Clearly, $\sim$ is reflexive and symmetric. It is also transitive.
For if $a * \phi = b * \psi$ and $b * \zeta = c * \eta$, then
$$
a*\phi*\zeta = b * \psi * \zeta = b*\zeta*\psi = c*\eta*\psi,
$$
and so $a\sim c$.

Say that a \textit{decent cone} is a subset $C$ of $\scr B$ such that:
\begin{enumerate}
    \item[(DC1)] if $a\sim b$, then $a\in C$ if and only if $b\in C$
    \item[(DC2)] if $a,b\in C$ and $\lambda,\mu \ge 0$, then $\lambda a + \mu b \in C$.
    \item[(DC3)] if $a$ is everywhere non-negative, then $a\in C$
    \item[(DC4)] if $a$ is everywhere non-positive, and not identically zero, then $a\not\in C$.
\end{enumerate}

Note that it follows from (DC1) and the fact that $\scr M$ contains a convolutional identity, namely $\delta_0$, that $C$ is closed under convolution with members of $\scr M$.

Let $-C = \{ -a : a\in C \}$.

\begin{lem}\label{lem:cone} Every decent cone $C_0$ can be extended to a decent cone $C\supseteq C_0$ such that
    $C\cup -C=\scr B$ and $C\cap -C = C_0\cap -C_0$.
\end{lem}

I will call any such cone $C$ a \textit{completion} of $C_0$.

\begin{proof}[Proof of Lemma~\ref{lem:cone}]
By Zorn's Lemma, all we need to do is to suppose that $C_0$ is a decent cone such that
$C_0\cup -C_0 \subset \scr B$, and show there is a decent cone $C_1$ such that $C_0\subset C_1$
and $C_1\cap -C_1 = C_0\cap -C_0$.

To that end, fix $c\notin C_0\cup -C_0$. Let $C_1$ be the set of all functions $b$ such that
$b*\phi = c*\psi + d$ for some $d \in C_0$, $\phi\in \scr M$ and $\psi\in \scr M^* = \scr M\cup \{0\}$.
Since $0\in C_0$ by (DC3) and $c \in C_1$, we have $C_0 \subset C_1$.

For (DC1), suppose $b\in C_1$ is such that $b*\phi = c*\psi + d$, with $d\in C_0$,
$\phi\in \scr M$ and $\psi\in \scr M^*$,
and suppose that $b*\zeta = b'*\zeta'$ for $\zeta,\zeta'\in\scr M$.
Then $b'*\zeta'*\phi=b*\zeta*\phi=b*\phi*\zeta=c*\psi*\zeta+d*\zeta$,
using the commutativity and associativity properties of our convolution.
Since $C_0$ is closed under convolutions with members of $\scr M$, we have $d*\zeta\in C_0$, and $b'\in C_1$.

For (DC2), note that clearly $C_1$ is closed under multiplication by a non-negative scalar, so all
we need to do is to show that it is closed under addition.
Suppose $b*\phi = c*\psi + d$ and $b'*\phi' = c*\psi' + d'$, with $d,d'\in C_0$, $\phi,\phi'\in \scr M$
and $\psi,\psi'\in \scr M^*$.
Then:
\begin{align*}
    (b+b')*\phi*\phi' &= b*\phi*\phi' + b'*\phi'*\phi \\
        &= c*\psi*\phi'+d*\phi' + c*\psi'*\phi+d'*\phi \\
        &= c*(\psi*\phi'+\psi'*\phi)+(d*\phi'+d'*\phi),
\end{align*}
and so $b+b'$ is in $C_1$ as $C_0$ is closed under convolution with members of $\scr M$ and under addition.

Condition (DC3) holds for $C_1$ as it holds for $C_0$.

Next we need to show that $C_1\cap -C_1 \subseteq C_0\cap -C_0$  (the other inclusion is trivial). Suppose that $b \in C_1\cap-C_1$. Thus,
$b*\phi = c*\psi + d$ and $-b*\phi' = c*\psi' + d'$ for some
$d,d'\in C_0$, $\phi,\phi'\in\scr M$ and $\psi,\psi'\in\scr M^*$.
Then
$$
    c*\psi*\phi' = b*\phi*\phi' - d*\phi'
$$
and
$$
    c*\psi'*\phi = -b*\phi'*\phi - d'*\phi.
$$
Since $\phi*\phi'=\phi'*\phi$, adding these two equalities we get
$$
    c*(\psi*\phi'+\psi'*\phi) = -d*\phi' - d'*\phi.
$$
If at least one of $\psi$ and $\psi'$ is not identically zero, it follows that $-c\sim d*\phi'+d'*\phi$, and
hence that $c \in -C_0$, contrary to our assumptions. If both $\psi$ and $\psi'$ are identically zero, then
it follows $d$ and $d'$ are identically zero. In that case $b*\phi = d$ and $-b*\phi' = d'$, so $b$ and $-b$
are both members of $C_0$, and hence $b\in C_0\cap -C_0$.

Finally, suppose that $a$ is non-positive but not identically zero. Then $-a\in C_1$ by (DC3). Hence if $a\in C_1$,
we have $a\in C_1\cap -C_1 = C_0\cap -C_0$, which contradicts the fact that $C_0$ satisfies (DC4).
\end{proof}

\begin{lem}\label{lem:cone2} Suppose that $C$ is a subset of $\scr B$ that is
    closed under right convolution with members of $\scr M$ (i.e., if $c\in C$ and $\psi\in\scr M$,
    then $c*\psi\in C$), is closed under addition, and satisfies (DC3) and (DC4).
    Then $C^* = \{ c \in \scr B : \exists \phi \in \scr M(c*\phi \in C) \}$ is a decent cone.
\end{lem}

\begin{proof}[Proof of Lemma~\ref{lem:cone2}]
    To check condition (DC1), suppose $c*\psi = c'*\psi'$ for $c\in\scr B$, $c'\in C^*$ and $\psi,\psi'\in \scr M$.
    Fix $\phi\in\scr M$ such that $c'*\phi\in C$. Then $c*\psi*\phi = c'*\psi'*\phi=c'*\phi*\psi'$.
    Since $C$ is closed under right convolution with members of $\scr M$, we have $c'*\phi*\psi'\in C$,
    and so $c\in C^*$ as desired.

    Note that $C$ is closed under multiplication by non-negative scalars
    since multiplication by a positive scalar $\lambda$ is just convolution with $\lambda\delta_0$,
    while $0\in C$. It
    follows that $C^*$ is closed under multiplication by non-negative scalars.
    To check (DC2), we need only check that $C^*$ is closed under addition.
    Suppose $a,b\in \scr B$,  $a',b'\in  C$ and $\phi,\phi',\psi,\psi'\in \scr M$ are such that
$
    a * \phi = a' * \phi'
$
    and
$
    b * \psi = b' * \psi'.
$
    Then:
\begin{equation*}
\begin{split}
    (a + b) * \phi*\psi &= a*\phi*\psi + b * \phi*\psi \\
            &= a'*\phi'*\psi + b * \psi*\phi \\
            &= a'*\phi'*\psi + b'*\psi'*\phi.
\end{split}
\end{equation*}
    The right-hand-side is in $C$, so $a+b$ must be in $C^*$.

    That $C^*$ satisfies (DC3) follows from the fact that $C\subseteq C^*$ as $a*\delta_0 = a$.
    It remains to show that (DC4) is satisfied. Suppose $a \in C^*$ is non-positive but not identically zero.
    Then $a*\phi \in C$ for some $\phi\in\scr M$. But $a*\phi$ will also be non-positive but not identically zero,
    contradicting the fact that $C$ satisfies (DC4).
\end{proof}

Let $1_A$ be the indicator function of a set $A$, i.e., the function that is $1$ on $A$ and $0$ outside $A$.

\begin{lem}\label{lem:cone3} Let $C$ be a decent cone such that $C\cup -C=\scr B$.
    Stipulate that $A\lessapprox B$ just in case
    $1_B-1_A \in C$. Then $\lessapprox$ is a total regular qualitative $\Z$-invariant probability.
\end{lem}

\begin{proof}[Proof of Lemma~\ref{lem:cone3}]
Reflexivity of $\lessapprox$ follows from the fact that $0\in C$. Transitivity follows
immediately from the fact that a decent cone is closed under addition.
Additivity follows from the fact that $1_B-1_A = 1_{B-A}-1_{A-B}$.
Since every non-negative function in $\scr B$ is a member of $C_0\subseteq C$, it follows
that $\varnothing\lessapprox A$ for all $A$.

To prove regularity, suppose $A$ is nonempty. If $A\lessapprox\varnothing$, then $1_\varnothing - 1_A = -1_A$ is in $C$,
contradicting (DC4).

Totality follows from the fact that $C\cup -C = \scr B$.

Observe
that $(1_B - 1_A) * \delta_x = 1_{x+B} - 1_{x+A}$. But a decent cone is invariant under right convolution
with members of $\scr M$ by (DC1), so $x+A \lessapprox x+B$ if and only if $A\lessapprox B$, so we have
$\Z$-invariance.
\end{proof}

West's result on the existence of a total regular qualitative $\Z$-invariant probability then follows
by letting $C_0$ be the collection of all functions in $\scr B$ that are everywhere non-negative,
letting $C$ be the completion of $C_0^*$, and applying the above lemmas.

Now say that a function $a\in\scr B$ has property $X_1$ provided that it is negative only in finitely
many places and that $\sum a=\sum_{n=-\infty}^\infty a_n\ge 0$. Say it has property $X_2$ provided
that for infinitely many $n>0$ we have $a(n)>0$ and there are only finitely many $n>0$ such that $a(n)<0$.
The sum of two functions with property $X_i$ has property $X_i$, for $i=1,2$, and the sum of a function
with $X_1$ and a function with property $X_2$ has property $X_2$. Having property $X_i$ is closed under
right-convolution with a member of $\scr M$: in the case of $X_2$, this uses the fact that
$\sum (a*\phi)=(\sum a)(\sum \phi)$ if $a$ is negative in only finitely many places and $\phi\in\scr M$.

Say that a function has property $X$ provided it has $X_1$ or $X_2$. Having property $X$ is thus closed under
addition and right-convolution with $\scr M$.

\begin{proof}[Proof of Proposition~\ref{prp:skew2}]
    Let $C_0$ be the set of functions $a\in\scr B$ that have property $X$.
    Note that any non-negative function has property $X_1$ and hence is in $C_0$,
    and the only non-positive function that can have $X$ is zero.

    Let $C$ be a completion of $C_0^*$. Define $\lessapprox$ as in Lemma~\ref{lem:cone3}. This will be a regular total
    $G$-invariant qualitative probability.

    Let's examine $C_0^* \cap -C_0^*$. Suppose $a \in C_0^* \cap -C_0^*$. Then
    $a*\phi$ and $-a*\psi$ have property $X$ for some $\phi,\psi\in\scr M$. Hence, so do
    $a*\phi*\psi$ and $-a*\psi*\phi$ since property $X$ is closed under right $\scr M$ convolution.
    Note that $\phi*\psi=\psi*\phi$. Let $b=a*\phi*\psi$. Then $b$ and $-b$ both have property $X$. There are three cases to consider: both functions have $X_1$,
    both have $X_2$, and one has $X_1$ while the other has $X_2$. It is clearly impossible
    that both $b$ and $-b$ have $X_2$. And if a function has $X_1$, then its negation is
    positive in only finitely many places, and so that negation cannot have $X_2$. Thus, the
    remaining case is where both functions have $X_1$. The only way this can be is if both
    functions sum to zero and are finitely supported. But if $a*\phi*\psi$ is finitely
    supported, so is $a$, and if one sums to zero, so does the other since $\sum a = (\sum a)(\sum \phi)(\sum \psi)$.
    Thus, any function in $C_0^*\cap -C_0^*$ is finitely supported and sums to zero. Conversely, any finitely
    supported function that sums to zero has $X_1$ and so does its negation.

    So, $C_0^*\cap -C_0^*=C\cap -C$ is the set of all finitely supported functions that sum to zero.

    Suppose $B$ has infinitely many positive members and $A$ does not. Then $1_B-1_A$ has property $X_2$,
    and hence is in $C$, and so $A\lessapprox B$. If we also had $B\lessapprox A$, we would have
    $1_B-1_A$ in $C\cap -C$, which would require $1_B-1_A$ to be finitely supported, which it's not.
    Thus, $A<B$.

    Finally, suppose $A$ and $B$ are finite and of the same cardinality. Then $1_B-1_A$ and $1_A-1_B$ both
    have $X_1$, and so $A\approx B$.
\end{proof}

Now suppose $C$ is a decent cone such that $C\cup -C=\scr B$.
For any $a,b\in\scr B$, define $a\lessapprox b$ provided that $b-a\in C$. This is a total vector space preorder on $\scr B$
(i.e., it's a total preorder such that if $a\lessapprox b$, $\lambda\ge 0$ and $c$ in the vector space, then $\lambda a + c \le \lambda b + c$).

We now need a useful lemma about total vector space preorders.
Let
$$
    \gamma(a,b) = \inf \{ \alpha \in \R : a \lessapprox \alpha b \}
$$
if $b\not\approx 0$. If $b\approx 0$ and $a\not\approx 0$, let $\gamma(a,b)=\infty$. And make $\gamma(a,b)$ be undefined
if $a\approx 0$ and $b\approx 0$.

\begin{lem}\label{lem:gamma}
    If $c\not\approx 0$, then $\gamma(a+b,c)=\gamma(a,c)+\gamma(b,c)$. Moreover, if
    $0\lessapprox a,b,c$, then
$$
    \gamma(a,c) = \gamma(a,b)\gamma(b,c)
$$
    whenever the right-hand-side is well-defined. If $b\not\approx 0$, then $\gamma(b,b)=1$. Finally, if $0\lessapprox a$ and $0<b$, then $0\le \gamma(a,b)$.
\end{lem}

\begin{proof}[Proof of Proposition~\ref{prp:skew2p}]
    Let $C$ be as in the proof of Proposition~\ref{prp:skew2}. By abuse of notation, use $\lessapprox$ both for the
    qualitative probability in that proof and for our vector space preorder. We then have  $A\lessapprox B$
    if and only if $1_A\lessapprox 1_B$. It follows from the non-negativity and regularity of the qualitative
    probability that $0\lessapprox 1_A$ for every $A$, with the inequality being strict if $A$ is nonempty.

    Define $c(A,B)=\gamma(1_A,1_B)$. It follows from Lemma~\ref{lem:gamma} that this is a coherent exchange rate.
    Then $P_c$ will be a full conditional probability.
    Moreover, $\gamma(a*\phi,b*\phi)=\gamma(a,b)$ for all $\phi\in\scr M$ by (DC1). Letting $\phi=\delta_x$, we see that
    $c$ satisfies the weak $\Z$-invariance condition $c(x+A,x+B)=c(A,B)$. It follows that $P=P_c$ is weakly
    $\Z$-invariant by Lemma~\ref{lem:cP-inv}.

    Next, suppose $m$ and $n$ are distinct integers (the case $m=n$ is trivial).
    Let $a_\alpha = 1_{\{m\}} - \alpha \cdot 1_{\{m,n\}}$. If $\alpha>1/2$, then $-a_\alpha$
    has property $X_1$. It follows that $-a_\alpha \in C$, so $1_{\{m\}} \lessapprox \alpha  \cdot1_{\{m,n\}}$.
    Thus, $\gamma(1_{\{m\}},1_{\{m,n\}}) \le 1/2$. On the other hand, if $\alpha<1/2$, then
    $a_\alpha$ has property $X_1$. It follows that $a_\alpha\in C$. The only way we could have
    $-a_\alpha$ in $C$ as well is if $a_\alpha\in C\cap -C$, which according to the proof of Proposition~\ref{prp:skew2}
    would require that $a_\alpha$ sum to zero, which it does not for $\alpha<1/2$. So, we do not have
    $1_{\{m\}} \lessapprox \alpha \cdot 1_{\{m,n\}}$, and hence $\gamma(1_{\{m\}},1_{\{m,n\}})\ge 1/2$. Thus,
    $\gamma(1_{\{m\}},1_{\{m,n\}})=1/2$, and it follows that $P(\{m\}\mid \{m,n\})=1/2$. Swapping $m$ and $n$ we get
    that $P(\{m\}\mid \{m,n\})=P(\{n\}\mid \{m,n\})$.

    Now, suppose that $A$ has only finitely many positive integers and $B$ has infinitely many. Then
    $P(A\mid A\cup B)=c(A,A\cup B)=\gamma(1_A,1_{A\cup B})$. Fix any $\alpha>0$. Then $\alpha\cdot 1_{A\cup B}-1_A$
    has property $X_2$, and hence is in $C$, so $1_A \lessapprox \alpha\cdot 1_{A\cup B}$. Thus, $\gamma(1_A,1_{A\cup B}) \le 0$,
    and by Lemma~\ref{lem:gamma} we have $\gamma(1_A,1_{A\cup B})=0$.
    Thus, $P(A\mid A\cup B)=0$, and so we must have $P(B\mid A\cup B)=1$ by finite additivity.
\end{proof}

\begin{proof}[Proof of Lemma~\ref{lem:gamma}]
For convenience, write:
$$
    \gamma^+(a,b) = \gamma(a,b)
$$
and
$$
    \gamma^-(a,b) = \sup \{ \alpha \in R : \alpha b \lessapprox a \},
$$
whenever $b\not\approx 0$, with the expressions undefined otherwise.
Observe that $a\lessapprox\alpha b$ if and only if $(-\alpha) b\lessapprox -a$, so we have the duality
$$
    \gamma^+(a,b) = -\gamma^-(-a,b).
$$
It is very easy to see that $\gamma^+(\cdot,b)$ is subadditive:
$$
    \gamma^+(a+a',b)\le\gamma^+(a,b)+\gamma^+(a',b)
$$
and that $\gamma^-(\cdot,b)$ is superadditive:
$$
    \gamma^-(a+a',b)\ge\gamma^-(a,b)+\gamma^-(a',b).
$$

In particular, if $b\not\approx 0$, we have $0=\gamma^+(0,b)\le\gamma^+(a,b)+\gamma^+(-a,b)$ for all $a$.
Hence, $-\gamma^+(-a,b)\le\gamma^+(a,b)$. Suppose $-\gamma^+(-a,b)<\gamma^+(a,b)$. Choose $\alpha \in (-\gamma^+(-a,b),\gamma^+(a,b))$.
Since $\alpha<\gamma^+(a,b)$, we do not have $a\lessapprox \alpha b$, and hence we have $\alpha b < a$ by totality. Thus,
we have $-a < -\alpha b$, so $\gamma^+(-a,b)\le -\alpha$, or $\alpha \le -\gamma^+(-a,b)$, a contradiction to the choice of $\alpha$. Hence,
$$
    \gamma^+(-a,b)=-\gamma^+(a,b),
$$
as long as $b\not\approx 0$. The same is true for $\gamma^-$ by the earlier proved duality between $\gamma^+$ and $\gamma^-$.
Therefore,
$$
    \gamma^+(a,b)=-\gamma^+(-a,b)=\gamma^-(a,b),
$$
if $b\not\approx 0$. It follows that if $b\not\approx 0$, then $\gamma(\cdot,b)$ is both subadditive and superadditive, and hence it is additive.

Now, if $0\lessapprox a$ and $0<b$, then $\gamma(a,b)=\gamma^+(a,b)=\gamma^-(a,b)$. If $0<a$ and $0\approx b$, then
$\gamma(a,b)=\infty$. And if $0\approx a$ and $0\approx b$, then $\gamma(a,b)$ is undefined.

I claim that if $0\lessapprox a,b,c$, and $\gamma(a,b)\gamma(b,c)$ is defined, then $\gamma(a,c)=\gamma(a,b)\gamma(b,c)$.

To see this, consider first the case where $0<b$ and $0<c$. Then if $a \lessapprox \alpha b$ and $b\lessapprox \beta c$, we
have $a\lessapprox \alpha\beta c$. It follows that $\gamma^+(a,c) \le \gamma^+(a,b)\gamma^+(b,c)$. Moreover, if
$\alpha b \lessapprox a$ and $\beta c\lessapprox b$, then $\alpha\beta c \lessapprox a$, and so $\gamma^-(a,c) \ge \gamma^-(a,b)\gamma^-(b,c)$.
But since $\gamma^+(u,v)=\gamma^-(u,v)$ whenever $0<v$, we thus have  $\gamma(a,c)=\gamma(a,b)\gamma(b,c)$.

Next, consider the case where $0\approx c$. Then $\gamma(b,c)$ is either undefined or equal to infinity. If it is undefined,
we don't need to prove anything. So suppose it's equal to infinity. For $\gamma(a,b)\gamma(b,c)$ to be defined, we must
have $\gamma(a,b)$ defined and strictly positive. This will happen only if $0<a$. But in that case $\gamma(a,c)=\infty$, and
so we have $\gamma(a,c)=\infty=\gamma(a,b)\gamma(b,c)$.

For the last case, suppose that $0<c$ but $0\approx b$. For $\gamma(a,b)$ to be defined, we must have $0<a$.
In that case $\gamma(a,b)=\infty$. For the product $\gamma(a,b)\gamma(b,c)$ to be defined, we must have $0<\gamma(b,c)$,
which is impossible if $b\approx 0$.

Next, suppose $0\not\approx b$. Then $\gamma^+(b,b) \le 1$ and $\gamma^-(b,b) \ge 1$ by reflexivity of $\lessapprox$, hence
$\gamma(b,b)=\gamma^+(b,b)=\gamma^-(b,b)=1$.

Finally, if $0\lessapprox a$ and $0<b$, then $0\le \gamma^-(a,b)=\gamma(a,b)$.
\end{proof}


\begin{thebibliography}{99}
\bibitem{Armstrong89} Armstrong, Thomas E. (1989). ``Invariance of full conditional probabilities under group actions'',
In: R. D. Mauldin, R. M. Shortt and C. E. Silva (eds.), \textit{Measure and Measurable Dynamics: Proceedings of a Conference
in Honor of Dorothy Maharam Stone, held September 17--19, 1987}. 1--22. Providence RI: American Mathematical Society.
\bibitem{ArmstrongSudderth89} Armstrong, Thomas E. and Sudderth, William D. (1989). ``Locally coherent rates of exchange'',
\textit{Annals of Statistics} 17:1394--1408.
\bibitem{BHW18} Benci, V., Horsten, L., and Wenmackers, S. 2018. ``Infinitesimal probabilities'',
    \textit{British Journal for the Philosophy of Science} 69:509--552.
\bibitem{BW69} Bernstein, Allen R., and Wattenberg, Frank. 1969. ``Non-standard Measure Theory.'' In {\em Applications of Model Theory
of Algebra, Analysis, and Probability}, ed.\ W. A. J. Luxemberg, 171--185.  New York: Holt, Rinehart and Winston.
\bibitem{Blumenthal40}  Blumenthal, L.M. 1940. ``A paradox, a paradox, a most ingenious paradox'', \textit{American
    Mathematical Monthly} 47:346.
\bibitem{Cornulier13} Cornulier, Yves. 2013. ``Answer to `Totally right preorderable groups'{}'',
    \textit{Mathoverflow}
    \url{http://mathoverflow.net/questions/147141/totally-right-preorderable-groups}
\bibitem{Exel} Exel, Ruy. 2017. \textit{Partial Dynamical Systems, Fell Bundles and Applications}. American Mathematical Society: Providence, RI.
\bibitem{Hajek03} H\'ajek, Alan. 2003. ``What conditional probability could not be'', \textit{Synthese} 137:273--323.
\bibitem{Howson17} Howson, C. 2017. ``Regularity and infinitely tossed coins'', \textit{European Journal for Philosophy of Science} 7:97–-102.
\bibitem{Just88} Just, Winfried. 1988. ``A bounded paradoxical subset of the plane'', \textit{Bulletin of the Polish Academy of Sciences -- Mathematics} 36:1--3.
\bibitem{KPS59} Kraft, C. H., Pratt, J. W., and Seidenberg, A. 1959. ``Intuitive probability on finite sets'', \textit{Annals of Mathematical
    Statistics} 30:408--419.
\bibitem{KLST71} Krantz, D. H., Luce, R. D., Suppes, P., and Tversky, B. 1971. \textit{Foundations of Measurement: Volume I: Additive and Polynomial Representations},
    San Diego: Academic Press.
\bibitem{Meehan} Meehan, Alexander. 2020. ``You say you want a revolution: On two notions of probabilistic independence'', manuscript.
\bibitem{Norton2018} Norton, John D. 2018. ``How to build an infinite lottery machine.''
    \textit{European Journal for the Philosophy of Science} 8:71--95.
\bibitem{Parker19} Parker, Matthew W. 2019. ``Symmetry arguments against regular probability: A reply to recent objections'',
    \textit{European Journal for Philosophy of Science} 9.
\bibitem{Pruss13} Pruss, Alexander R. 2013. ``Null probability, dominance and rotation'', \textit{Analysis} 73:682--685.
\bibitem{PrussInvariance} Pruss, Alexander R. 2013. ``Two kinds of invariance of full conditional probabilities'',
    \textit{Bulletin of the Polish Academy of Sciences -- Mathematics} 61:277--283.
\bibitem{Pruss14} Pruss, Alexander R. 2014. ``Regular probability comparisons imply the Banach-Tarski paradox'', \textit{Synthese} 191:3525--3540.
\bibitem{PrussPopper} Pruss, Alexander R. 2015. ``Popper functions, uniform distributions and infinite sequences of heads'',
    \textit{Journal of Philosophical Logic} 44:259--271.
\bibitem{Scarparo18} Scarparo, Eduardo. 2018. ``Characterizations of locally finite actions of groups on sets'',
    \textit{Glasgow Mathematical Journal} 60:285--288.
\bibitem{WT16} Tomkowicz, G., and Wagon, S. 2016. \textit{The Banach Tarski Paradox}, 2nd ed. Cambridge University Press: Cambridge.
\bibitem{West20} West, Harry. 2020. ``Answer to `Comparing sizes of sets of integers'{}'', \textit{Mathoverflow} \url{https://mathoverflow.net/questions/370690/comparing-sizes-of-sets-of-integers}.
\bibitem{Williamson07} Williamson, Timothy. 2007. ``How probable is an infinite sequence of heads?'' \textit{Analysis} 67:173--180.
    \end{thebibliography}
\end{document}